\documentclass[12pt,reqno]{amsart}

\newcommand\version{December 05, 2008}

\usepackage{amsmath,amsfonts,amsthm,amssymb,amsxtra}

\newtheorem{theorem}{Theorem}

\newtheorem{lemma}[theorem]{Lemma}

\theoremstyle{definition}

\theoremstyle{remark}

\renewcommand{\epsilon}{\varepsilon}

\newcommand{\N}{\mathbb{N}}

\newcommand{\R}{\mathbb{R}}

\newcommand{\Sph}{\mathbb{S}}

\begin{document}

\title[Low energy scattering --- \version]{A note on low energy scattering\\ for homogeneous long range potentials}

\author{Rupert L. Frank}
\address{Rupert L. Frank, Department of Mathematics,
Princeton University, Washington Road, Princeton, NJ 08544, USA}
\email{rlfrank@math.princeton.edu}

\thanks{\copyright\, 2008 by the author. This paper may be reproduced, in its entirety, for non-commercial purposes.}

\maketitle


The very interesting, recent paper \cite{DeSk2} is concerned with scattering at zero energy for long-range potentials. It is shown that if $V=-\alpha |x|^{-\mu}+W$ for some $\alpha>0$, $0<\mu<2$ and a sufficiently fast decaying, radial function $W$, then the scattering matrix for the pair $(-\Delta +V,-\Delta)$ at zero energy is given by
$$
\exp\left( -i \tfrac{\pi\mu}{2-\mu}\sqrt{-\Delta_{\Sph^{d-1}} + \left(\tfrac{d-2}2\right)^2}\right) + K
$$
for some compact operator $K$ (depending on $V$). Moreover, it is conjectured that if $V=-\alpha |x|^{-\mu}$ (that is, if $W\equiv 0$), then $K\equiv 0$. The purpose of this note is to prove this conjecture.

\begin{theorem}\label{sm}
Let $d\geq 2$, $0<\mu<2$ and $\alpha>0$. For $\lambda>0$ let $S(\lambda)$ be the scattering matrix for the pair $(-\Delta -\alpha|x|^{-\mu},-\Delta)$ at energy $\lambda$ in the sense of \cite{DeSk1}. Then
$$
\mathrm{s}-\lim_{\lambda\to 0+} S(\lambda) = \exp\left( -i \tfrac{\pi\mu}{2-\mu}\sqrt{-\Delta_{\Sph^{d-1}} + \left(\tfrac{d-2}2\right)^2}\right) \,.
$$
\end{theorem}

It was shown in \cite{DeSk1} that $S(0):=\mathrm{s}-\lim_{\lambda\to 0+} S(\lambda)$ exists and is given through the phase shifts of the solutions of the energy zero equations corresponding to fixed angular momentum. To be more precise, for any $l\in\N_0$ there exists a function $f$ satisfying
\begin{align*}
-f''+\left( \frac{(l-1+d/2)^2 -1/4}{r^2} - \frac{\alpha}{r^\mu}\right) f & =0
\qquad \text{in}\  \R_+ \,, \\
\lim_{r\to0} r^{-l-\tfrac{d-1}2} f(r) & = 1 \,, \\
\lim_{r\to\infty} \left( r^{-\mu/4} f(r) -  C_l \sin\left(\frac{2\sqrt\alpha}{2-\mu} \ r^{\frac{2-\mu}2} + D_l \right) \right) & = 0
\end{align*}
for some $C_l>0$, and in terms of this function the action of $S(0)$ on a spherical harmonic $Y$ of order $l$ is given by
$$
S(0) Y = \exp\left(i2\left(D_l + \tfrac\pi 4 (d-3+2l)\right)\right) Y \,.
$$
(Note that the $2\pi$-ambiguity in the definition of $D_l$ does not affect the formula for $S(0)$.) In view of these facts and recalling that the Laplacian acts on spherical harmonics of order $l$ as multiplication by $l(l+d-2)$, Theorem \ref{sm} will follow if we can prove that $D_l=-\frac{\pi (d-2+2l)}{2(2-\mu)} + \frac\pi4$ modulo $2\pi$. This equality is the assertion of the following lemma applied to $\nu=l+\tfrac{d-2}2$.

\begin{lemma}
 Let $\nu\geq 0$, $0<\mu<2$ and $\alpha>0$. The function
\begin{equation}
\label{eq:sol}
f(r) := \Gamma(\tfrac{2\nu}{2-\mu}+1) \left(\frac{2-\mu}{\sqrt\alpha} \right)^{\tfrac{2\nu}{2-\mu}}\ r^{1/2} \ J_{\frac{2\nu}{2-\mu}} ( \tfrac{2\sqrt\alpha}{2-\mu} \ r^{\frac{2-\mu}2} )
\end{equation}
satisfies
\begin{align}
 \label{eq:eq}
-f''+\left(\frac{\nu^2-1/4}{r^2} - \frac{\alpha}{r^\mu}\right) f & =0
\qquad \text{in}\  \R_+ \,, \\
\label{eq:bc}
\lim_{r\to0} r^{-\nu-1/2} f(r) & = 1 \,, \\
\label{eq:asymp}
\lim_{r\to\infty} \left( r^{-\mu/4} f(r) -  C \sin\left(\frac{2\sqrt\alpha}{2-\mu} \ r^{\frac{2-\mu}2} -\frac{\pi\nu}{2-\mu} + \frac\pi4 \right) \right) & = 0
\end{align}
with $C := \pi^{-1/2} \Gamma(\tfrac{2\nu}{2-\mu}+1) \left((2-\nu)/\sqrt\alpha \right)^{\tfrac{2\nu}{2-\mu}+\tfrac12} >0$.
\end{lemma}

\begin{proof}
We will show that the function $f$ given in \eqref{eq:sol} is the unique solution of the initial value problem \eqref{eq:eq} -- \eqref{eq:bc}. Then the asymptotics \eqref{eq:asymp} follow from the asymptotics
$$
J_{\tilde\nu} (s) = \sqrt{\frac2{\pi s} } \left( \sin\left(s-\frac{\pi\tilde\nu}{2} + \frac\pi4 \right) + o(1) \right) \,,
\qquad s\to\infty\,,
$$
of Bessel functions \cite[(9.2.1)]{AbSt}.

Let $f$ denote any solution of the initial value problem \eqref{eq:eq} -- \eqref{eq:bc} and define $g$ by
$$
f(r) =: r^{1/2} g(b r^{\frac{2-\mu}2} ) \,,
\qquad
b:=\frac{2\sqrt\alpha}{2-\mu} \,.
$$
This definition is motivated by the asymptotics \eqref{eq:asymp}: If we want to arrive at a function which behaves asymptotically like an inverse square root times an oscillating function -- which is the asymptotic behavior of any Bessel function --, then \eqref{eq:asymp} suggests to look at $r^{-1/2}f(r)$ as a function of $r^{\frac{2-\mu}2}$. This is essentially what we call $g$, and we are about to show that $g$ is indeed a Bessel function.

A short computation shows that equation \eqref{eq:eq} in terms of $g$ becomes
$$
g'' + s^{-1}g' - \left(\frac{2\nu}{2-\mu}\right)^2 s^{-2} g + g = 0 \,,
$$
which is Bessel's equation with parameter $\tilde\nu:=2\nu/(2-\mu)$. Hence $g$ is a linear combination of $J_{\tilde\nu}$ and $Y_{\tilde\nu}$. Boundary condition \eqref{eq:bc} in terms of $g$ becomes
$$
\lim_{s\to0} s^{-\tilde\nu} g(s) = b^{-\tilde\nu} \,.
$$
Since $J_{\tilde\nu} (s) \sim (s/2)^{\tilde\nu}/\Gamma(\tilde\nu+1)$ for all $\tilde\nu\geq 0$ and $Y_{\tilde\nu}(s) \sim -(1/\pi)\Gamma(\tilde\nu) (s/2)^{-\tilde\nu}$ for $\tilde\nu>0$, resp. $Y_0(s) \sim (2/\pi)\ln s$ as $s\to 0$ \cite[(9.1.7--9)]{AbSt}, we conclude that
$$
g(s) = \Gamma(\tilde\nu+1) (2/b)^{\tilde\nu} J_{\tilde\nu} (s) \,.
$$
This proves the lemma.
\end{proof}

\subsection*{Acknowledgements}

The author wishes to thank R. Seiringer for useful discussions. Support through DFG grant FR 2664/1-1 and U.S. NSF grant PHY 06 52854 is gratefully acknowledged.


\bibliographystyle{amsalpha}

\end{document}